\documentclass[letterpaper, 10 pt, conference]{ieeeconf}  
\usepackage{graphicx}
\usepackage{epstopdf}

\usepackage{relsize}
\usepackage{amsmath,amssymb,amsfonts,mathtools}
\usepackage{nicematrix, nicefrac}
\usepackage{algorithmic}
\usepackage{textcomp}
\usepackage{xcolor}

\usepackage{booktabs}
\usepackage{graphicx}
\usepackage{epsfig} 
\usepackage{times}  
\usepackage{siunitx}
\NewCommandCopy\OldSI\SI
\RenewDocumentCommand\SI{ O{} m m }{%
  \nobreak%
  \OldSI[#1]{#2}{#3}%
  \nobreak%
}

\usepackage{tikz}
\usepackage{xstring}
\usepackage{forloop}
\usepackage{colortbl}
\usepackage{soul}
\usepackage{array}

\let\labelindent\relax  
\usepackage{enumitem}

\usepackage{multirow} 
\usepackage{multicol}
\usepackage{arydshln}
\usepackage{pgfgantt}

\usepackage{textcomp}
\usepackage{rotating}
\usepackage{float}
\usepackage{algorithmic}
\usepackage{algorithm}
\usepackage[edges]{forest}

\newcommand{\dd}{\mathrm{d}}

\DeclareMathOperator{\grad}{grad}
\DeclareMathOperator{\Hess}{Hess}

\newcommand{\norm}[1]{\left\|#1\right\|}

\newcommand{\R}{\mathbb R}

\DeclareMathOperator{\rank}{rank}

\makeatletter
  \renewcommand*\env@matrix[1][*\c@MaxMatrixCols c]{%
	\hskip -\arraycolsep
	\let\@ifnextchar\new@ifnextchar
	\array{#1}}
\newcommand*\bigcdot{\mathpalette\bigcdot@{1.5}}
\newcommand*\bigcdot@[2]{\mathbin{\vcenter{\hbox{\scalebox{#2}{$\m@th#1\cdot$}}}}}
\makeatother

\newcounter{excounter}[section]

\newcounter{rmcounter}

\newcounter{dfcounter}

\newcounter{thmcounter}

\newcounter{corcounter}[thmcounter]

\newcounter{prcounter}

\newcounter{clcounter}

\usepackage{xurl}
\usepackage{import}

\IEEEoverridecommandlockouts                              
\usepackage{amsmath, amssymb}
\usepackage{comment}

\usepackage{amsthm}

\usepackage{tikz-cd}
\tikzcdset{
  arrow style=tikz,
  diagrams={>={Stealth[length=4pt,width=3.8pt]}}
}

\newtheorem{problemEnv}{Problem}
\newenvironment{problem}[1][]{\begin{problemEnv}}{
\end{problemEnv}}

\newtheorem{remarkEnv}{Remark}
\newenvironment{remark}[1][]{\begin{remarkEnv}}{
\end{remarkEnv}}

\newtheorem{theoremEnv}{Theorem}
\newtheorem{lemmaEnv}[theoremEnv]{Lemma}
\newtheorem{propositionEnv}[theoremEnv]{Proposition}
\newtheorem{corollaryEnv}[theoremEnv]{Corollary}

\newenvironment{theorem}[1][]{\begin{theoremEnv}}
{\hfill$ $\end{theoremEnv}}

\newenvironment{lemma}[1][]{\begin{lemmaEnv}}
{\hfill$ $\end{lemmaEnv}}

\newenvironment{proposition}[1][]{\begin{propositionEnv}}
{\hfill$ $\end{propositionEnv}}

\newtheorem{definitionEnv}{Definition}

\usepackage{url}
\usepackage{subcaption}
\usepackage{epstopdf}
\usepackage{xcolor}

\usepackage{xcolor}

\usepackage{algorithmic,algorithm}

\usepackage{cite}
\makeatletter
\let\NAT@parse\undefined
\makeatother
\usepackage[colorlinks=true,allcolors=blue]{hyperref}

\title{\LARGE \bf
Modeling and Control of a Unicycle Robot on Manifolds
}

\author{Adeel Akhtar$^{1}$ and Mohamed Al Lawati$^{2}$
\thanks{This work is supported by NJIT's startup funds.}
\thanks{$^{1}$ Department of Mechanical and Industrial Engineering, New Jersey Institute of Technology, NJ, USA. {\tt\small adeel.akhtar@njit.edu}}
\thanks{$^{2}$ Department of Mechanical and Industrial Engineering, Sultan Qaboos University, Muscat, Oman. {\tt\small mlawati@squ.edu.om}}
}

\begin{document}

\maketitle
\thispagestyle{empty}
\pagestyle{empty}


\begin{abstract}
This paper studies path-following control for a kinematic unicycle evolving
on a smooth manifold embedded in \(\mathbb{R}^3\). The objective is to
stabilize a one-dimensional path on the manifold without imposing a temporal
parameterization, thereby distinguishing the problem from trajectory
tracking. Working in a geometric framework, we propose two feedback laws:
a static controller for prescribed nonzero translational speed and a dynamic
controller based on a dynamic extension of the translational speed. By
lifting the path to the unit tangent bundle of the surface, we characterize
the corresponding path-following manifold and establish its control
invariance. We then show that the proposed controllers render the
path-following manifold locally exponentially stable under suitable
regularity conditions. A numerical simulation illustrates the proposed controller. Code and animations are publicly available
at~{\footnotesize\url{https://gradslab.github.io/UnicycleOnManifolds/}}.
\end{abstract}

\section{Introduction}


Feedback control of kinematic unicycles is a classical topic in nonlinear
control, with extensive literature on stabilization and trajectory
tracking; see~\cite{Bro83,Ast96,CheJia15,RehReyVan19} and the references
therein. While stabilization concerns regulation to a target configuration,
trajectory tracking prescribes the motion through a
time-parameterized reference trajectory~\cite{tariq2026relaxed,RodVel22}.

In many applications, however, the desired motion is specified only as a
geometric path, without an associated temporal parameterization; see~\cite{AkhWasNie2013Journal} and the references therein. This leads
to the path-following problem, where the objective is to stabilize the path without prescribing the temporal evolution along it. Path following for kinematic
unicycles and related nonlinear systems has received considerable attention in planar settings~\cite{ElhMag08,SkjKok2004,NieMag04}. In this setting, the
target is treated as a geometric subset of the state space, and the
closed-loop objective naturally involves both stabilization of the path and
invariance of the induced motion on it~\cite{AkhNie2011,NieMag04}.

In contrast, feedback control of a kinematic unicycle evolving on a curved
manifold embedded in \(\R^3\) remains relatively unexplored. Existing works
in this setting have primarily addressed motion planning and open-loop
control on general or specific surfaces~\cite{MccReyReh17}.
The problem arises, for example, in planetary exploration and autonomous
navigation over mountainous or otherwise non-flat terrain. A key difficulty
is that both the vehicle kinematics and the control objective must respect
the geometry of the underlying manifold.

This paper develops path-following controllers for a kinematic unicycle
evolving on a smooth manifold \(\mathsf{M}\) embedded in \(\R^3\). First,
we derive the kinematic model directly in a geometric framework on the unit
tangent bundle of the surface. Second, we construct two feedback laws for path
following of a geometric curve \(\mathsf{C}\subset\mathsf{M}\): a static
controller for prescribed nonzero translational speed and a dynamic
controller that additionally permits shaping the longitudinal evolution
along the path.

\begin{figure}
    \centering
    \includegraphics[width=0.9\linewidth]{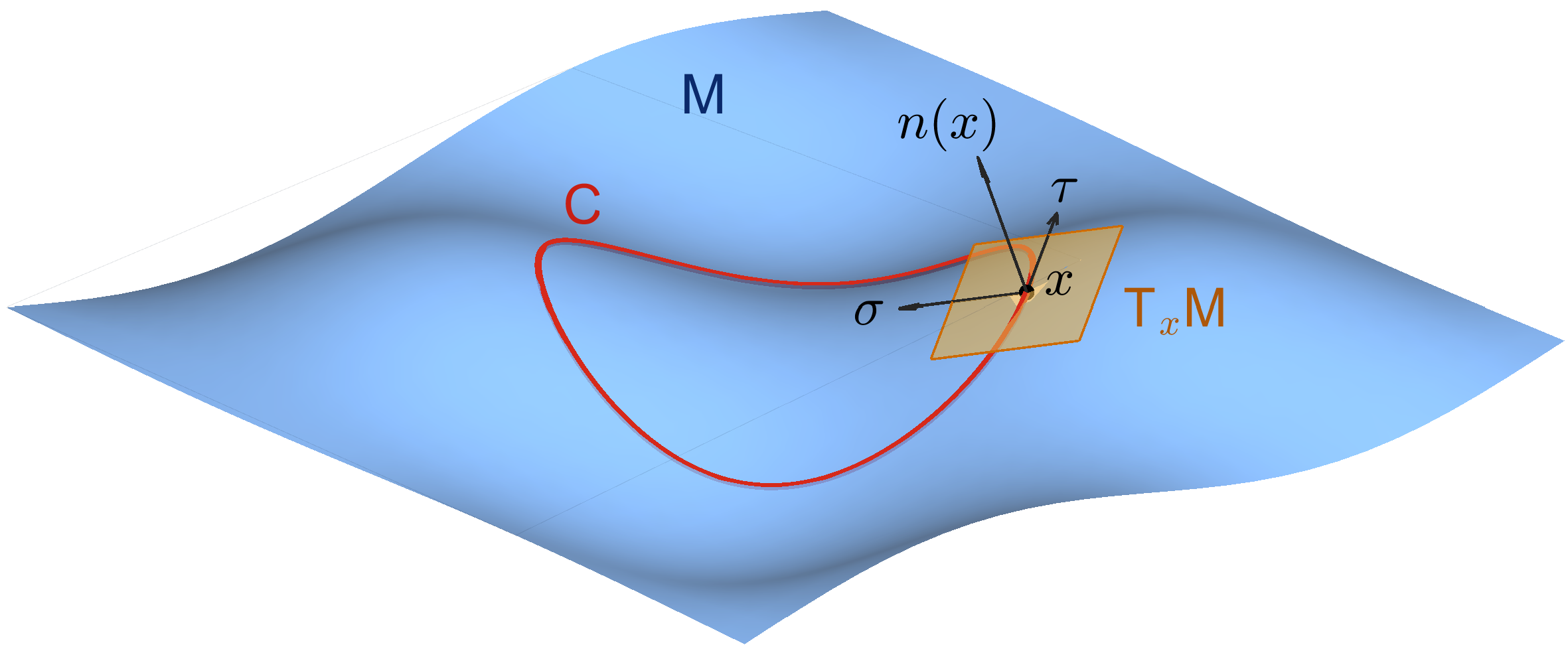}
    \caption{A target path \(\mathsf{C}\) on the manifold \(\mathsf{M}\).
    At each \(x\in\mathsf{M}\), the vectors \(\{\tau,\sigma\}\) form an
    orthonormal basis of \(\mathsf{T}_x\mathsf{M}\), and \(n(x)\) is the
    unit normal vector.}
    \label{fig:M}
\end{figure}

Our formulation is related to virtual holonomic constraint (VHC) and
virtual nonholonomic constraint (VNC) approaches, but addresses a
structurally different plant~\cite{ConCosMag2017,SimStrBloCol2023}. Existing geometric VHC/VNC theories are formulated for second-order
mechanical control systems, whereas the system considered here is a
first-order kinematic unicycle evolving on the unit tangent bundle, with
velocity-level inputs. Our setting therefore requires a geometric
construction of the lifted path-following manifold tailored to this
surface-constrained kinematic system, together with a dynamic extension for
shaping the longitudinal evolution along the path.

The main contributions are as follows:
\begin{enumerate}
    \item for the fixed-speed controller, we prove local exponential
    stability of the path-following manifold and path invariance; see
    Theorem~\ref{thm:fixed_v};

    \item for the variable-speed controller, we prove local exponential
    stability of the extended path-following manifold, path invariance, and
    exponential tracking of a prescribed longitudinal coordinate-rate law;
    see Theorem~\ref{thm:dyn_fdbk}.
\end{enumerate}

\noindent{\textbf{Notation:}}
We use the standard basis \((e_1,e_2,e_3)\) of \(\R^3\). All vectors are column vectors unless stated otherwise, and \(x^\top\) denotes the transpose of \(x\in\R^n\). For a \(C^2\) real-valued function \(f\), its Hessian is denoted by \(\Hess f\). For a smooth manifold \(\mathsf{M}\), \(\mathsf{T}_x\mathsf{M}\) and \(\mathsf{T}_x^*\mathsf{M}\) denote the tangent and cotangent spaces at \(x\in\mathsf{M}\), respectively, and \(\mathsf{T}\mathsf{M}\) denotes the tangent bundle, i.e., the disjoint union of the tangent spaces \(\mathsf{T}_x\mathsf{M}\), \(x\in\mathsf{M}\). Throughout the paper, \(\simeq\) denotes a canonical isomorphism; in particular, \(\mathsf{T}_y\R^n\simeq\R^n\) for every \(y\in\R^n\).


\section{Mathematical Preliminaries}\label{sec:preliminaries}
We first introduce the differential-geometric notation and identities used
throughout the paper; see~\cite{BulLew2004} for further background.
Let \(\mathsf{M}\) and \(\mathsf{N}\) be smooth manifolds, and let
\(F:\mathsf{M}\to\mathsf{N}\) be smooth. Its differential at
\(x\in\mathsf{M}\) is denoted by
\(
\dd F(x):\mathsf{T}_x\mathsf{M}\to \mathsf{T}_{F(x)}\mathsf{N}.
\)
Let \((\mathsf{M},g)\) be a Riemannian manifold. For a smooth function
\(f:\mathsf{M}\to\R\), the Riemannian gradient \(\grad f(x)\in\mathsf{T}_x\mathsf{M}\)
is defined by
\(
\dd f(x)[v]=g_x(\grad f(x),v),
\text{ for all } v\in\mathsf{T}_x\mathsf{M}.
\)
Let \(\nabla\) denote the Levi--Civita connection of \(g\). For
\(h\in C^2(\mathsf{M})\), its Hessian is the symmetric \((0,2)\)-tensor
defined by
\(
\Hess h(X,Y) := X\big(\dd h[Y]\big)-\dd h\big[\nabla_XY\big] = g\big(\nabla_X(\grad h),Y\big),
\)
for all smooth vector fields \(X,Y\) on \(\mathsf{M}\).
If \(x(\cdot)\) is a differentiable curve on \(\mathsf{M}\) and
\(v(\cdot)\in \mathsf{T}_{x(\cdot)}\mathsf{M}\) is a differentiable vector
field along \(x(\cdot)\), then
\begin{equation}\label{eq:hessian_curve_identity}
\begin{aligned}
\frac{d}{dt}\big(\dd h(x(t))[v(t)]\big)
&=
\Hess h(x(t))\big(\dot x(t),v(t)\big)\\
&\quad+
\dd h(x(t))\big[\nabla_{\dot x(t)}v(t)\big].
\end{aligned}
\end{equation}
When \(\mathsf{M}=\R^3\), we use the canonical identification
\(\mathsf{T}_x\R^3\simeq\R^3\) and the standard Euclidean metric. Hence, for smooth \(f:\R^3\to\R\),
\(
\dd f(x)[v]=\langle \grad f(x),v\rangle
\)
for all \(v\in\R^3\).
\section{Modeling}
\label{sec:modeling}
We develop the surface-constrained kinematic model for a unicycle robot,
building on the geometric modeling viewpoint of~\cite{MccReyReh17}. Let \(\phi:\R^3\to\R\) be \(C^r\), \(r\ge 2\), and define
\(
\mathsf{M}:=\{x\in\R^3:\phi(x)=0\}.
\)
Assume that \(0\) is a regular value of \(\phi\). Then \(\mathsf{M}\) is a two-dimensional embedded \(C^r\) submanifold of \(\R^3\), endowed with the metric induced by the Euclidean inner product. For each \(x\in\mathsf{M}\),
\(
\mathsf{T}_x\mathsf{M}
=
\{w\in\R^3:\dd\phi(x)[w]=0\}.
\)
The robot position is \(x\in\mathsf{M}\), and its heading is a unit tangent vector \(\tau\in\mathsf{T}_x\mathsf{M}\). Hence the configuration manifold is the unit tangent bundle
\begin{equation}
\label{eq:unit-tangent-bundle}
\mathsf{Q}:=U\mathsf{T}\mathsf{M}
=
\{(x,\tau):x\in\mathsf{M},\ \tau\in\mathsf{T}_x\mathsf{M},\ \|\tau\|=1\}.
\end{equation}
Since \(\mathsf{M}\) is a regular level set, a global unit normal field is given by
\[
n(x):=\frac{\grad\phi(x)}{\|\grad\phi(x)\|},
\qquad x\in\mathsf{M}.
\]
For \((x,\tau)\in\mathsf{Q}\), we  define
\(
\sigma:=n(x)\times \tau.
\)
Then \(\sigma\in\mathsf{T}_x\mathsf{M}\), and \(\{\tau,\sigma\}\) and \(\{\tau,\sigma,n(x)\}\) are orthonormal bases of \(\mathsf{T}_x\mathsf{M}\) and \(\R^3\), respectively; see Fig.~\ref{fig:M}.

Let \(v,\omega:\R_{\ge0}\to\R\) denote the translational and angular
velocity inputs. For an admissible motion
\(t\mapsto(x(t),\tau(t))\in\mathsf{Q}\), the no-side-slip kinematic
constraint requires \(\dot x(t)=v(t)\tau(t)\).
Moreover, \(\|\tau(t)\|=1\) and
\(\langle n(x(t)),\tau(t)\rangle=0\), so
%
%
\begin{align*}
\langle \dot\tau(t),\tau(t)\rangle &=0, \\
\langle \dot\tau(t),n(x(t))\rangle
&=
-v(t)\big\langle \tau(t),\dd n(x(t))[\tau(t)]\big\rangle.
\end{align*}
Thus, with \(\sigma(t):=n(x(t))\times\tau(t)\) and
\(
\langle \dot\tau(t),\sigma(t)\rangle=\omega(t),
\)
the decomposition of \(\dot\tau(t)\) in the basis
\(\{\tau(t),\sigma(t),n(x(t))\}\) gives
\begin{equation*}
\begin{aligned}
\dot\tau(t)
&=
\omega(t)\big(n(x(t))\times\tau(t)\big)\\
&\quad
-v(t)\big\langle \tau(t),\dd n(x(t))[\tau(t)]\big\rangle n(x(t)).
\end{aligned}
\end{equation*}
Suppressing the explicit dependence on \(t\), the unicycle kinematics on \(\mathsf{M}\), with state \((x,\tau)\in\mathsf{Q}\) and input \((v,\omega)\in\R^2\), are
\begin{subequations}\label{eq:geo_model}
\begin{align}
\dot x &= v\tau, \label{eq:geo_model_1}\\
\dot\tau
&=
\omega\big(n(x)\times\tau\big)
-
v\big\langle \tau,\dd n(x)[\tau]\big\rangle n(x).
\label{eq:geo_model_2}
\end{align}
\end{subequations}
If \((x(0),\tau(0))\in\mathsf{Q}\), then
\[
\frac{d}{dt}\phi(x(t))=0,
\;\;
\frac{d}{dt}\|\tau(t)\|^2=0,
\;\;
\frac{d}{dt}\langle n(x(t)),\tau(t)\rangle=0.
\]
Hence \(x(t)\in\mathsf{M}\), \(\tau(t)\in\mathsf{T}_{x(t)}\mathsf{M}\), and \(\|\tau(t)\|=1\) for all \(t\), so \eqref{eq:geo_model} defines a control system on \(\mathsf{Q}\).

The term
\(-v\langle\tau,\dd n(x)[\tau]\rangle n(x)\)
captures the variation of the tangent spaces along the motion, whereas \(\omega\) rotates the heading within \(\mathsf{T}_x\mathsf{M}\) along \(n(x)\times\tau\).
\begin{remark}
For \(\phi(x)=x_3\), the manifold is \(\mathsf{M}=\{x\in\R^3:x_3=0\}\), with \(n(x)=e_3\) and \(\dd n(x)=0\). Then \eqref{eq:geo_model} reduces to \(\dot x=v\tau\), \(\dot\tau=\omega(e_3\times\tau)\). Writing \(\tau=[\cos\theta,\sin\theta,0]^\top\) yields \(\dot x_1=v\cos\theta\), \(\dot x_2=v\sin\theta\), and \(\dot\theta=\omega\), which is the classical planar unicycle~\cite{AkhNie2011,NieMag04,ElhMag08}.
\end{remark}
\section{Problem Formulation}
\label{sec:problem}

Consider the unicycle model \eqref{eq:geo_model} evolving on the configuration manifold \(\mathsf{Q}=U\mathsf{T}\mathsf{M}\), i.e., the unit tangent bundle defined in~\eqref{eq:unit-tangent-bundle}.
Let \(h_1:\mathsf{M}\to\R\) be a \(C^r\) function, with \(r\ge 2\), and assume that \(0\) is a regular value of \(h_1\). Define the target path
\begin{equation}\label{eq:C_def_pf}
\mathsf{C}:=\{x\in\mathsf{M}:h_1(x)=0\}.
\end{equation}
Then \(\mathsf{C}\) is a one-dimensional embedded \(C^r\) submanifold of \(\mathsf{M}\), and
\begin{equation}\label{eq:TxC_kernel_pf}
\mathsf{T}_x\mathsf{C}=\ker \dd h_1(x),
\qquad \forall x\in \mathsf{C}.
\end{equation}
Since the state of the vehicle includes both position and heading, the relevant target set is not merely the curve \(\mathsf{C}\subset \mathsf{M}\), but its unit tangent lift to \(\mathsf{Q}\). Accordingly, define the path-following manifold
\begin{equation}\label{eq:Gamma_def_pf}
\begin{aligned}
\mathsf{\Gamma}
&:=
\{(x,\tau)\in\mathsf{Q}:x\in\mathsf{C},\ \tau\in\mathsf{T}_x\mathsf{C}\}\\
&=
\{(x,\tau)\in\mathsf{Q}:h_1(x)=0,\ \dd h_1(x)[\tau]=0\}.
\end{aligned}
\end{equation}
Thus, \(\mathsf{\Gamma}\) consists of all configurations for which the vehicle lies on the target path and its heading is tangent to that path. Since no temporal parameterization is prescribed for \(\mathsf{C}\), the objective is path following rather than trajectory tracking.

The first problem considered in this paper is the following.

\begin{problem}[Fixed-speed geometric path following]\label{prob:fixed_speed}
Given a prescribed constant translational speed
\(
v(t)\equiv \bar v,
\) for
\(
\bar v\in\R\setminus\{0\},
\)
design a feedback law \(\omega=\kappa(x,\tau)\), defined on a neighborhood
of \(\mathsf{\Gamma}\), such that, for the resulting closed-loop system,
\begin{enumerate}[label=\textbf{P1.\arabic*}]
\item \label{itm:P11} the manifold \(\mathsf{\Gamma}\) is forward invariant;
\item \label{itm:P12} the manifold \(\mathsf{\Gamma}\) is locally exponentially stable.
\end{enumerate}
\end{problem}

For the variable-speed design, let \(h_2:\mathsf{M}\to\R\) be another \(C^r\) function, and define
\begin{equation}\label{eq:h_pair_pf}
h:=(h_1,h_2):\mathsf{M}\to\R^2.
\end{equation}
Introduce the singular set
\begin{equation}\label{eq:Sigma_h_pf}
\Sigma_h:=\{x\in\mathsf{M}:\rank \dd h(x)<2\},
\end{equation}
and the corresponding regularity domain
\(
\mathcal{O}_h:=\mathsf{M}\setminus \Sigma_h.
\)
For each \(x\in\mathcal{O}_h\), the differential
\(
\dd h(x):\mathsf{T}_x\mathsf{M}\to\R^2
\)
is a linear isomorphism. In particular, for \(x\in \mathsf{C}\cap\mathcal{O}_h\), the restriction of \(\dd h_2(x)\) to \(\mathsf{T}_x\mathsf{C}=\ker \dd h_1(x)\) is nonzero; hence \(h_2\) defines a local longitudinal coordinate along \(\mathsf{C}\cap\mathcal{O}_h\).

To obtain a square input-output structure, we introduce the dynamic extension
\(
z:=v,
\;
\dot z=\bar u,
\)
where \(\bar u\in\R\) is a new input. The corresponding extended system on \(\mathsf{Q}\times\R\) is
\begin{subequations}\label{eq:ext_dyn_pf}
\begin{align}
\dot x &= z\tau, \\
\dot \tau &= \omega\big(n(x)\times \tau\big)-z\big\langle \tau,\dd n(x)[\tau]\big\rangle n(x), \\
\dot z &= \bar u.
\end{align}
\end{subequations}
Next, we define
\begin{equation}\label{eq:Xext_pf}
\mathcal{X}:=
\{(x,\tau,z)\in\mathsf{Q}\times\R:x\in\mathcal{O}_h,\ z\neq 0\},
\end{equation}
and the corresponding extended path-following manifold
\begin{equation}\label{eq:Gamma_ext_def_pf}
\mathsf{\Gamma}_{\mathrm{ext}}
:=
\{(x,\tau,z)\in\mathcal{X}:(x,\tau)\in\mathsf{\Gamma}\}.
\end{equation}
Let \(\upsilon_d:\R\to\R\) be a \(C^1\) function satisfying
\begin{equation}\label{eq:upsilon_d_assump_pf}
|\upsilon_d(s)|\ge v_{\min}>0,
\qquad \forall s\in\R,
\end{equation}
for some constant \(v_{\min}>0\). 
The second control objective is to enforce the longitudinal evolution law
\(
\dot h_2=\upsilon_d(h_2)
\)
along the path. Accordingly, define the longitudinal error
\begin{equation}\label{eq:e2_pf}
e_2:=\dot h_2-\upsilon_d(h_2).
\end{equation}
This leads to the following problem.

\begin{problem}[Dynamic path following with longitudinal coordinate-rate shaping]\label{prob:dynamic_speed}
Design a feedback law \((\bar u,\omega)=\kappa_{\mathrm{ext}}(x,\tau,z)\) for the dynamically extended system~\eqref{eq:ext_dyn_pf} such that, for the resulting closed-loop system on \(\mathcal{X}\),
\begin{enumerate}[label=\textbf{P2.\arabic*}]
\item \label{itm:P21} the manifold \(\mathsf{\Gamma}_{\mathrm{ext}}\) is forward invariant;
\item \label{itm:P22} the manifold \(\mathsf{\Gamma}_{\mathrm{ext}}\) is locally exponentially stable;
\item \label{itm:P23} the longitudinal error \(e_2\) converges to zero exponentially.
\end{enumerate}
\end{problem}

Problem~\ref{prob:dynamic_speed} prescribes the evolution of the longitudinal coordinate \(h_2\), not necessarily the translational speed \(v\) itself. In the special case where \(h_2\) is chosen as a local arc-length coordinate along \(\mathsf{C}\), the law \(\dot h_2=\upsilon_d(h_2)\) coincides with a prescribed signed speed profile along the path.

\section{Feedback Control Scheme 1}\label{sec:fixed_v}

In this section, the translational speed is fixed at a prescribed nonzero constant value
\(
v(t)\equiv \bar v.
\)
Accordingly, the unicycle model \eqref{eq:geo_model} reduces to the single-input control system
\begin{subequations}\label{eq:geo_model_fixed_v}
\begin{align}
\dot{x} &= \bar v\,\tau, \label{eq:geo_model_fixed_v_1}\\
\dot{\tau} &= \omega\big(n(x)\times \tau\big)
-
\bar v\,\big\langle \tau,\dd n(x)[\tau]\big\rangle n(x). \label{eq:geo_model_fixed_v_2}
\end{align}
\end{subequations}
The control input is the angular velocity \(\omega\in\R\).

Let \(y:\mathsf{Q}\to\R\) denote the scalar output
\(y(x,\tau):=h_1(x)\).
Along a trajectory, let \(\sigma(t):=n(x(t))\times\tau(t)\).
The next proposition gives the input-output relations needed for the feedback design.

\begin{proposition}\label{prop:rel_deg_two}
Consider the system \eqref{eq:geo_model_fixed_v} and the output $y = h_1(x)$. Then, along every differentiable trajectory of \eqref{eq:geo_model_fixed_v},
\begin{subequations}\label{eq:y_derivatives}
\begin{align}
\dot y(t) &= \bar v\,\dd h_1(x(t))[\tau(t)], \label{eq:y_dot}\\
\ddot y(t) &= \bar v^{\,2}\,\Hess h_1(x(t))\big(\tau(t),\tau(t)\big)
+
\bar v\,\omega(t)\,\dd h_1(x(t))[\sigma(t)].
\label{eq:y_ddot}
\end{align}
\end{subequations}
\end{proposition}

\begin{proof}
The identity \eqref{eq:y_dot} follows from the chain rule and \eqref{eq:geo_model_fixed_v_1}. Differentiating \eqref{eq:y_dot} and using \eqref{eq:hessian_curve_identity}, we obtain
\[
\ddot y(t)
=
\bar v\,\Hess h_1(x(t))\big(\dot x(t),\tau(t)\big)
+
\bar v\,\dd h_1(x(t))\big[\nabla_{\dot x(t)}\tau(t)\big].
\]
Since the Levi--Civita connection is the tangential projection of the ambient derivative, \eqref{eq:geo_model_fixed_v_2} implies
\(
\nabla_{\dot x(t)}\tau(t)=\omega(t)\sigma(t).
\)
Substituting this identity together with \(\dot x(t)=\bar v\,\tau(t)\) gives \eqref{eq:y_ddot}.
\end{proof}

Next, we define scalar functions \(a,b:\mathsf{Q}\to\R\) by
\begin{equation}\label{eq:a_b_defs}
\begin{aligned}
a(x,\tau)&:=\bar v\,\dd h_1(x)\big[n(x)\times \tau\big],\\
b(x,\tau)&:=\bar v^{\,2}\,\Hess h_1(x)(\tau,\tau).
\end{aligned}
\end{equation}
Then \eqref{eq:y_ddot} can be written as
\begin{equation}\label{eq:yddot_ab}
\ddot y = a(x,\tau)\,\omega + b(x,\tau).
\end{equation}
Since \(\bar v\neq 0\), \eqref{eq:Gamma_def_pf} and
Proposition~\ref{prop:rel_deg_two} imply that
\[
\mathsf{\Gamma}
=
\{(x,\tau)\in\mathsf{Q}: y(x,\tau)=0,\ \dot y(x,\tau)=0\}.
\]
%
The next lemma shows that the decoupling coefficient does not vanish on \(\mathsf{\Gamma}\).

\begin{lemma}\label{lem:a_nonzero}
For every \((x,\tau)\in\mathsf{\Gamma}\), one has
\(
a(x,\tau)\neq 0.
\)
Consequently, there exists an open neighborhood \(\mathcal U\subset\mathsf{Q}\) of \(\mathsf{\Gamma}\) such that
\(
a(x,\tau)\neq 0,
\text{ for all } (x,\tau)\in\mathcal U.
\)
\end{lemma}

\begin{proof}
Let \((x,\tau)\in\mathsf{\Gamma}\). Then \(x\in\mathsf{C}\) and \(\tau\in\mathsf{T}_x\mathsf{C}\). Since \(0\) is a regular value of \(h_1\), one has
\(
\mathsf{T}_x\mathsf{C}=\ker \dd h_1(x).
\)
Hence \(\dd h_1(x)[\tau]=0\).
Let \(\sigma:=n(x)\times \tau\). Since \(\{\tau,\sigma\}\) is an orthonormal basis of \(\mathsf{T}_x\mathsf{M}\), the vector \(\sigma\) spans the orthogonal complement of \(\mathsf{T}_x\mathsf{C}\) in \(\mathsf{T}_x\mathsf{M}\). On the other hand, \(\grad h_1(x)\in \mathsf{T}_x\mathsf{M}\) is orthogonal to \(\ker \dd h_1(x)=\mathsf{T}_x\mathsf{C}\), and therefore \(\grad h_1(x)\) is collinear with \(\sigma\). Hence
\[
\dd h_1(x)[\sigma]
=
\langle \grad h_1(x),\sigma\rangle
=
\pm \|\grad h_1(x)\|.
\]
Since \(0\) is a regular value of \(h_1\), one has \(\grad h_1(x)\neq 0\). Thus \(\dd h_1(x)[\sigma]\neq 0\), and, because \(\bar v\neq 0\),
\[
a(x,\tau)=\bar v\,\dd h_1(x)[\sigma]\neq 0.
\]
The second claim follows from continuity of \(a\).
\end{proof}

On the neighborhood \(\mathcal U\), Lemma~\ref{lem:a_nonzero} allows us to define the feedback law
\begin{equation}\label{eq:omega_fixed_v}
\omega
=
\frac{-\,b(x,\tau)+\nu}{a(x,\tau)},
\end{equation}
where \(\nu\in\R\) is an auxiliary input. Substituting \eqref{eq:omega_fixed_v} into \eqref{eq:yddot_ab} yields
\(\ddot y=\nu\).

Hence, with
\(
\xi:=
(\xi_1,\xi_2)^\top
:=
(y,\dot y)^\top,
\)
the transverse dynamics take the form
\begin{equation}\label{eq:xi_dynamics}
\dot\xi
=
\begin{bmatrix}
0 & 1\\
0 & 0
\end{bmatrix}\xi
+
\begin{bmatrix}
0\\
1
\end{bmatrix}\nu.
\end{equation}
Choose
\begin{equation}\label{eq:nu_fixed_v}
\nu=-k_1\xi_1-k_2\xi_2,
\qquad
k_1,k_2>0,
\end{equation}
which yields
\begin{equation}\label{eq:xi_closed_loop}
\dot\xi
=
\begin{bmatrix}
0 & 1\\
-k_1 & -k_2
\end{bmatrix}\xi,
\end{equation}
whose origin is exponentially stable.

The next lemma shows that \(\xi=(y,\dot y)\) defines local transverse coordinates to \(\mathsf{\Gamma}\).

\begin{lemma}\label{lem:Gamma_local_coords}
Define \(\Psi:\mathsf{Q}\to\R^2\) by
\[
\Psi(x,\tau)
:=
\begin{bmatrix}
h_1(x)\\[1mm]
\bar v\,\dd h_1(x)[\tau]
\end{bmatrix}.
\]
Then \(\Psi^{-1}(0)=\mathsf{\Gamma}\), and \(\dd\Psi(x,\tau)\) has rank \(2\) for every \((x,\tau)\in\mathsf{\Gamma}\). Consequently, for every \(\bar q\in\mathsf{\Gamma}\), there exist an open neighborhood \(\mathcal W\subset\mathcal U\) of \(\bar q\), an open set \(\mathcal V\subset\R\), and a smooth coordinate chart
\(
\Phi=(\Psi,\eta):\mathcal W\to \R^2\times \mathcal V
\)
such that
\(
\mathsf{\Gamma}\cap\mathcal W
=
\{q\in\mathcal W:\Psi(q)=0\}.
\)
\end{lemma}

\begin{proof}
The identity \(\Psi^{-1}(0)=\mathsf{\Gamma}\) follows directly from \eqref{eq:Gamma_def_pf} and the fact that \(\bar v\neq 0\).
We can write \(\Psi=(\psi_1,\psi_2)^\top\), where
\(\psi_1(x,\tau):=h_1(x)\) and
\(\psi_2(x,\tau):=\bar v\,\dd h_1(x)[\tau]\).
Let \(\pi:\mathsf{Q}\to\mathsf{M}\), \(\pi(x,\tau)=x\), denote the canonical projection. Then \(\psi_1=h_1\circ \pi\). It follows from~\cite[Definition 3.38]{BulLew2004} that since \(\pi\) is a submersion and \(0\) is a regular value of \(h_1\), one has
\(
\dd\psi_1(x,\tau)\neq 0,
\text{for all } (x,\tau)\in\mathsf{\Gamma}.
\)
%
%
%
For \((x,\tau)\in\mathsf{\Gamma}\), let
\(\sigma:=n(x)\times\tau\) and \(G:=(0,\sigma)\).
The curve
\(c(s):=(x,\cos(s)\tau+\sin(s)\sigma)\) lies in
\(\mathsf{Q}\) and satisfies \(\dot c(0)=G\); hence
\(G\in\mathsf{T}_{(x,\tau)}\mathsf{Q}\).
Since \(\psi_1\) depends only on \(x\),
\(\dd\psi_1(x,\tau)[G]=0\), whereas
\[
\dd\psi_2(x,\tau)[G]
=
\bar v\,\dd h_1(x)[\sigma]
=
a(x,\tau).
\]

%
By Lemma~\ref{lem:a_nonzero},
\(\dd\psi_2[G]=a(x,\tau)\neq0\) on \(\mathsf{\Gamma}\).
Together with \(\dd\psi_1\neq0\) and \(\dd\psi_1[G]=0\), this implies
that \(\dd\psi_1\) and \(\dd\psi_2\) are linearly independent.
Hence \(\rank\dd\Psi(x,\tau)=2\) on \(\mathsf{\Gamma}\).
Now fix \(\bar q\in\mathsf{\Gamma}\). Since \(\rank \dd\Psi(\bar q)=2=\dim\R^2\), the map \(\Psi\) is a submersion at \(\bar q\). Hence, by the submersion theorem~\cite[pg. 76]{BulLew2004}, there exist an open neighborhood \(\mathcal W\subset\mathcal U\) of \(\bar q\), an open set \(\mathcal V\subset\R\), and a smooth coordinate chart
\(
\Phi=(\Psi,\eta):\mathcal W\to \R^2\times \mathcal V.
\)
Since \(\Psi^{-1}(0)=\mathsf{\Gamma}\), it follows that
\(
\mathsf{\Gamma}\cap\mathcal W
=
\{q\in\mathcal W:\Psi(q)=0\}.
\)
This proves the final statement of the lemma.
\end{proof}

It remains to characterize the reduced dynamics on \(\mathsf{\Gamma}\).
Let \(\mathsf{C}_0\) be a connected component of \(\mathsf{C}\), and
let \(t_{\mathsf{C}}\) be a local smooth unit tangent field on
\(\mathsf{C}_0\). Then \(\mathsf{\Gamma}\) locally decomposes into the
two branches
%
\begin{equation}\label{eq:Gamma_pm}
\mathsf{\Gamma}_\pm
:=
\{(x,\tau)\in\mathsf{\Gamma}:\tau=\pm t_{\mathsf{C}}(x)\}.
\end{equation}
Along either branch, the position dynamics satisfy
\begin{equation}\label{eq:zero_dyn_x}
\dot x
=
\bar v\,\tau
=
\pm \bar v\,t_{\mathsf{C}}(x).
\end{equation}
%
If \(s\) denotes a local arc-length coordinate on \(\mathsf{C}_0\) compatible with the orientation defined by \(t_{\mathsf{C}}\), then \eqref{eq:zero_dyn_x} reduces to
\begin{equation}\label{eq:zero_dyn_s}
\dot s=\pm \bar v.
\end{equation}
Thus, on \(\mathsf{\Gamma}\), the vehicle moves along the target path with constant signed speed \(\pm \bar v\), equivalently with constant speed \(|\bar v|\).

The preceding discussion is summarized in the following theorem.

\begin{theorem}\label{thm:fixed_v}
Assume that \(\bar v\neq 0\), that \(0\) is a regular value of \(h_1\), and that \(\mathsf{C}=h_1^{-1}(0)\) is the target path. Let \(\mathcal U\subset\mathsf{Q}\) be the neighborhood of \(\mathsf{\Gamma}\) provided by Lemma~\ref{lem:a_nonzero}. Then the feedback law \eqref{eq:omega_fixed_v}, \eqref{eq:nu_fixed_v} is well defined on \(\mathcal U\). Moreover, the following statements hold:
\begin{enumerate}[label=(\roman*)]
\item for every closed-loop trajectory taking values in \(\mathcal U\), the transverse variable \(\xi=(y,\dot y)^\top\)
evolves according to the linear system \eqref{eq:xi_closed_loop},

\item the manifold \(\mathsf{\Gamma}\) is forward invariant and locally exponentially stable in \(\mathcal U\), and

\item on each local branch \(\mathsf{\Gamma}_\pm\), the reduced dynamics are given by \eqref{eq:zero_dyn_x}, or equivalently by \eqref{eq:zero_dyn_s}.
\end{enumerate}
\end{theorem}

\begin{proof}
Statement~(i) follows from \eqref{eq:xi_dynamics} and the choice \eqref{eq:nu_fixed_v}, since the matrix in \eqref{eq:xi_closed_loop} is Hurwitz for \(k_1,k_2>0\).
For statement~(ii), first note that, because \(\bar v\neq 0\),
\[
\mathsf{\Gamma}
=
\{(x,\tau)\in\mathsf{Q}:h_1(x)=0,\ \dd h_1(x)[\tau]=0\}
=
\{\xi=0\}.
\]
Since \(\xi=0\) is an exponentially stable equilibrium of
\eqref{eq:xi_closed_loop}, \(\mathsf{\Gamma}\) is forward invariant.
Moreover, Lemma~\ref{lem:Gamma_local_coords} shows that \(\xi\) defines
local transverse coordinates to \(\mathsf{\Gamma}\); hence
\(\mathsf{\Gamma}\) is locally exponentially stable.
Finally, let a closed-loop trajectory evolve on \(\mathsf{\Gamma}\). Then \(x(t)\in\mathsf{C}\) and \(\tau(t)\in\mathsf{T}_{x(t)}\mathsf{C}\) for all \(t\). On any connected component \(\mathsf{C}_0\) and any local choice of smooth unit tangent field \(t_{\mathsf{C}}\), one has \(\tau(t)=\pm t_{\mathsf{C}}(x(t))\). Substituting this into \eqref{eq:geo_model_fixed_v_1} yields \eqref{eq:zero_dyn_x}, and \eqref{eq:zero_dyn_s} follows immediately. This proves statement~(iii).
\end{proof}

Theorem~\ref{thm:fixed_v} shows that the feedback law
\eqref{eq:omega_fixed_v}, \eqref{eq:nu_fixed_v} solves
Problem~\ref{prob:fixed_speed} locally. Statement~(ii) establishes
Objectives~\ref{itm:P11}--\ref{itm:P12}, while statement~(iii)
characterizes the induced motion on the path-following manifold.

\section{Feedback Control Scheme 2}\label{sec:dyn_fdbk}

%

We now consider the variable-speed design of
Problem~\ref{prob:dynamic_speed}, where the output
\(h=(h_1,h_2)\) consists of the transverse coordinate \(h_1\) and the
local longitudinal coordinate \(h_2\) on \(\mathcal O_h\).
Consider the dynamically extended system
\eqref{eq:ext_dyn_pf} on \(\mathcal X\), and define the corresponding
extended path-following manifold by
\begin{equation}\label{eq:Gammaext_dyn}
\mathsf{\Gamma}_{\mathrm{ext}}
:=
\{(x,\tau,z)\in\mathcal{X}:(x,\tau)\in\mathsf{\Gamma}\}.
\end{equation}
Since
\(\dot h_1=z\,\dd h_1(x)[\tau]\) and \(z\neq0\) on
\(\mathcal X\), \eqref{eq:Gamma_def_pf} gives
\begin{equation}\label{eq:Gammaext_h1}
\mathsf{\Gamma}_{\mathrm{ext}}
=
\{(x,\tau,z)\in\mathcal{X}:h_1(x)=0,\ \dot h_1(x,\tau,z)=0\}.
\end{equation}
Thus, \(\mathsf{\Gamma}_{\mathrm{ext}}\) is the zero-dynamics manifold
associated with the transverse output \(h_1\).

Define the vector output \(y_{\mathrm{ext}}:\mathsf{Q}\times\R\to\R^2\) by
\(y_{\mathrm{ext}}(x,\tau,z):=h(x)=(h_1(x),h_2(x))^\top\).
Along a trajectory \(t\mapsto (x(t),\tau(t),z(t))\), let
\(\sigma(t):=n(x(t))\times\tau(t)\).

The next proposition gives the input-output relations for the extended system.

\begin{proposition}\label{prop:dyn_ext_io}
Consider the extended system \eqref{eq:ext_dyn_pf} with output
\(y_{\mathrm{ext}}=h(x)\). Then, along every differentiable trajectory,

\begin{subequations}\label{eq:yext_derivatives}
\begin{align}
\dot h_i(t)
&=
z(t)\,\dd h_i(x(t))[\tau(t)],
\qquad i=1,2, \label{eq:yext_dot_i}\\
\ddot h_i(t)
&=
\dd h_i(x(t))[\tau(t)]\,\bar u(t) \nonumber\\
&\quad
+z(t)^2\,\Hess h_i(x(t))\big(\tau(t),\tau(t)\big) \nonumber\\
&\quad
+z(t)\omega(t)\,\dd h_i(x(t))[\sigma(t)],
\qquad i=1,2. \label{eq:yext_ddot_i}
\end{align}
\end{subequations}
Equivalently,
\begin{equation}\label{eq:yext_ddot_vec}
\ddot y_{\mathrm{ext}}
=
A(x,\tau,z)
\begin{bmatrix}
\bar u\\
\omega
\end{bmatrix}
+
\beta(x,\tau,z),
\end{equation}
where
\begin{equation}\label{eq:A_def_dyn}
A(x,\tau,z)
=
\begin{bmatrix}
\dd h_1(x)[\tau] & z\,\dd h_1(x)[n(x)\times \tau]\\[1mm]
\dd h_2(x)[\tau] & z\,\dd h_2(x)[n(x)\times \tau]
\end{bmatrix},
\end{equation}
and
\begin{equation}\label{eq:beta_def_dyn}
\beta(x,\tau,z)
=
z^2
\begin{bmatrix}
\Hess h_1(x)(\tau,\tau)\\[1mm]
\Hess h_2(x)(\tau,\tau)
\end{bmatrix}.
\end{equation}
\end{proposition}

\begin{proof}
For each \(i\in\{1,2\}\), \eqref{eq:yext_dot_i} follows from the chain rule and \(\dot x=z\tau\) in \eqref{eq:ext_dyn_pf}. Differentiating \eqref{eq:yext_dot_i} gives
\[
\ddot h_i(t)
=
\dot z(t)\,\dd h_i(x(t))[\tau(t)]
+
z(t)\frac{d}{dt}\big(\dd h_i(x(t))[\tau(t)]\big).
\]
Applying \eqref{eq:hessian_curve_identity}, we obtain
\begin{align}
\frac{d}{dt}\big(\dd h_i(x(t))[\tau(t)]\big)
&=
\Hess h_i(x(t))\big(\dot x(t),\tau(t)\big)\\
&+
\dd h_i(x(t))\big[\nabla_{\dot x(t)}\tau(t)\big].
\end{align}
Since the tangential component of \(\dot\tau(t)\) is \(\omega(t)\sigma(t)\), one has
\(
\nabla_{\dot x(t)}\tau(t)=\omega(t)\sigma(t).
\)
Substituting \(\dot x(t)=z(t)\tau(t)\) and \(\dot z(t)=\bar u(t)\) yields \eqref{eq:yext_ddot_i}. Equation \eqref{eq:yext_ddot_vec} follows by stacking the two components.
\end{proof}

We call \(A(x,\tau,z)\) the decoupling matrix of the extended system.

\begin{lemma}\label{lem:rel_deg_22}
For every \((x,\tau,z)\in\mathsf{Q}\times\R\) with \(x\in\mathcal{O}_h\), the matrix \(A(x,\tau,z)\) is invertible if and only if \(z\neq 0\). Equivalently, on \(\mathcal{X}\), the extended system \eqref{eq:ext_dyn_pf} has vector relative degree \(\{2,2\}\) with respect to the output \(y_{\mathrm{ext}}=h(x)\).
\end{lemma}

\begin{proof}
Fix \((x,\tau,z)\in\mathsf{Q}\times\R\) with \(x\in\mathcal{O}_h\), and let \(\sigma:=n(x)\times \tau\). Since \(\{\tau,\sigma\}\) is a basis of \(\mathsf{T}_x\mathsf{M}\), one has
\[
\det A(x,\tau,z)
=
z\,
\det
\begin{bmatrix}
\dd h_1(x)[\tau] & \dd h_1(x)[\sigma]\\[1mm]
\dd h_2(x)[\tau] & \dd h_2(x)[\sigma]
\end{bmatrix}.
\]
Because \(x\in\mathcal{O}_h\), the map
\(
\dd h(x):\mathsf{T}_x\mathsf{M}\to\R^2
\)
is a linear isomorphism. Hence the two vectors
\(
\dd h(x)[\tau],
\text{ and }
\dd h(x)[\sigma]
\)
form a basis of \(\R^2\), and therefore the determinant above is nonzero. It follows that \(\det A(x,\tau,z)\neq 0\) if and only if \(z\neq 0\). The relative-degree statement follows from \eqref{eq:yext_derivatives}.
\end{proof}

The next lemma shows that, on \(\mathcal{X}\), the outputs and their first derivatives define local coordinates.

\begin{lemma}\label{lem:Phi_ext_diffeo}
Define $\Phi_{\mathrm{ext}}:\mathcal{X}\to\R^4$ by
\[
\Phi_{\mathrm{ext}}(x,\tau,z)
:=
\begin{bmatrix}
h_1(x)\\[1mm]
\dot h_1(x,\tau,z)\\[1mm]
h_2(x)\\[1mm]
\dot h_2(x,\tau,z)
\end{bmatrix}
=
\begin{bmatrix}
h_1(x)\\[1mm]
z\,\dd h_1(x)[\tau]\\[1mm]
h_2(x)\\[1mm]
z\,\dd h_2(x)[\tau]
\end{bmatrix}.
\]
Then \(\Phi_{\mathrm{ext}}\) is a local diffeomorphism on \(\mathcal{X}\). In particular, for every \(\bar q\in\mathsf{\Gamma}_{\mathrm{ext}}\), there exists an open neighborhood \(\mathcal{W}\subset\mathcal{X}\) of \(\bar q\) such that, in the local coordinates
\[
\chi=\Phi_{\mathrm{ext}}(x,\tau,z)=
(\chi_1, \chi_2,\chi_3,\chi_4
)^\top,
\]
the manifold \(\mathsf{\Gamma}_{\mathrm{ext}}\cap\mathcal{W}\) is represented by
\(
\mathsf{\Gamma}_{\mathrm{ext}}\cap\mathcal{W}
=
\{q\in\mathcal{W}:\chi_1=0,\ \chi_2=0\}.
\)
\end{lemma}

\begin{proof}
By Lemma~\ref{lem:rel_deg_22}, the extended system \eqref{eq:ext_dyn_pf} has vector relative degree \(\{2,2\}\) on \(\mathcal{X}\). Since \(\dim\mathsf{M}=2\), one has \(\dim\mathsf{Q}=3\), and therefore
\(
\dim(\mathsf{Q}\times\R)=4=2+2.
\)
Thus, the system has full vector relative degree on \(\mathcal{X}\). By the standard exact-feedback-linearization normal-form theorem, the map formed by the outputs and their derivatives up to order \(r_i-1\), i.e., \(\Phi_{\mathrm{ext}}\), is a local diffeomorphism on \(\mathcal{X}\).
Moreover, by~\eqref{eq:Gammaext_h1},
\[
\mathsf{\Gamma}_{\mathrm{ext}}
=
\{(x,\tau,z)\in\mathcal{X}:h_1(x)=0,\ \dot h_1(x,\tau,z)=0\}.
\]
Hence, in the coordinates induced by \(\Phi_{\mathrm{ext}}\), the manifold \(\mathsf{\Gamma}_{\mathrm{ext}}\) is locally given by the equations \(\chi_1=0\) and \(\chi_2=0\).
\end{proof}

On \(\mathcal{X}\), Lemma~\ref{lem:rel_deg_22} allows us to define the
feedback law
\begin{equation}\label{eq:u_dyn_ext}
\begin{bmatrix}
\bar u\\
\omega
\end{bmatrix}
=
A(x,\tau,z)^{-1}\big(-\,\beta(x,\tau,z)+\mu\big),
\end{equation}
where \(\mu=(\mu_1,\mu_2)^\top\in\R^2\) is an auxiliary input.
Substituting \eqref{eq:u_dyn_ext} into \eqref{eq:yext_ddot_vec} yields
\(\ddot y_{\mathrm{ext}}=\mu\), i.e.,
\(\ddot h_1=\mu_1\) and \(\ddot h_2=\mu_2\).
To stabilize \(\mathsf{\Gamma}_{\mathrm{ext}}\), choose
\begin{equation}\label{eq:mu1_def}
\mu_1=-k_{11}h_1-k_{12}\dot h_1,
\qquad
k_{11},k_{12}>0,
\end{equation}
which gives
\(\ddot h_1+k_{12}\dot h_1+k_{11}h_1=0\).
We next specify \(\mu_2\) to shape the evolution of the longitudinal
coordinate \(h_2\). Recall the longitudinal error
\(e_2=\dot h_2-\upsilon_d(h_2)\). Its derivative satisfies
\(\dot e_2=\ddot h_2-\upsilon_d'(h_2)\dot h_2\). Choose
\begin{equation}\label{eq:mu2_def}
\mu_2=
\upsilon_d'(h_2)\dot h_2-k_2e_2,
\qquad
k_2>0.
\end{equation}
Then \(\dot e_2=-k_2e_2\), and hence \(\dot h_2\) tracks the prescribed
longitudinal coordinate-rate law \(\upsilon_d(h_2)\) exponentially.

The preceding construction is summarized as follows.

\begin{theorem}\label{thm:dyn_fdbk}
Consider the extended system \eqref{eq:ext_dyn_pf} on \(\mathcal{X}\), and apply the feedback law \eqref{eq:u_dyn_ext} with auxiliary input \(\mu\) given by \eqref{eq:mu1_def} and \eqref{eq:mu2_def}. Then the following statements hold for the resulting closed-loop system: 
\begin{enumerate}[label=(\roman*)]
\item the manifold \(\mathsf{\Gamma}_{\mathrm{ext}}\) is forward invariant and locally exponentially stable in \(\mathcal{X}\);

\item the longitudinal error \(e_2\) converges to zero exponentially;

\item if a closed-loop trajectory starts on \(\mathsf{\Gamma}_{\mathrm{ext}}\), then its position remains on \(\mathsf{C}\) for all future time for which the solution remains in \(\mathcal{X}\).
\end{enumerate}
\end{theorem}

\begin{proof}
Under \eqref{eq:u_dyn_ext}--\eqref{eq:mu2_def}, the closed-loop output
dynamics satisfy
\(
\ddot h_1+k_{12}\dot h_1+k_{11}h_1=0,
\;\;
\dot e_2=-k_2e_2.
\)
Hence \(e_2(t)=e_2(0)e^{-k_2t}\), which proves~(ii).
Moreover, by \eqref{eq:Gammaext_h1} and
Lemma~\ref{lem:Phi_ext_diffeo},
\((h_1,\dot h_1)\) defines local transverse coordinates to
\(\mathsf{\Gamma}_{\mathrm{ext}}\). Since the first equation above is
exponentially stable, the same transverse-stability argument as in the
proof of Theorem~\ref{thm:fixed_v} shows that
\(\mathsf{\Gamma}_{\mathrm{ext}}\) is forward invariant and locally
exponentially stable, proving~(i).
Forward invariance of
\(\mathsf{\Gamma}_{\mathrm{ext}}\) implies \(h_1(x(t))=0\), and hence
\(x(t)\in\mathsf{C}\), proving~(iii).
\end{proof}

\begin{remark}
    For every
\(\bar q=(\bar x,\bar\tau,\bar z)\in\mathsf{\Gamma}_{\mathrm{ext}}\),
one has \(\dd h_2(\bar x)[\bar\tau]\neq0\). Hence, by continuity,
there exists a neighborhood \(\mathcal W\subset\mathcal X\) of
\(\bar q\) on which \(\dd h_2(x)[\tau]\neq0\). Moreover,
\[
z(t)=
\frac{\upsilon_d(h_2(t))+e_2(0)e^{-k_2t}}
{\dd h_2(x(t))[\tau(t)]}.
\]
Since \(|\upsilon_d|\ge v_{\min}>0\),
the condition \(|e_2(0)|<v_{\min}\) guarantees \(z(t)\neq0\)
for all $t$ for which the trajectory remains in \(\mathcal W\).
Therefore, the singular set \(\{z=0\}\) is locally avoided.
\end{remark}

\section{Simulations}

We illustrate the variable-speed controller of
Section~\ref{sec:dyn_fdbk}; additional results, including the fixed-speed
case, detailed plots, animations, and source code are available on our
project webpage\footnote{\url{https://gradslab.github.io/UnicycleOnManifolds/}}
and GitHub\footnote{\url{https://github.com/gradslab/UnicycleOnManifolds}}.
We consider the sphere
\(\phi(x)=\norm{x}^2-4\), with
\(
h_1(x)=3x_1^2+x_2^2-3.24
\)
and local azimuthal coordinate
\(
h_2(x)=\operatorname{atan2}(x_2,x_1).
\)
The desired longitudinal coordinate rate is
\(\upsilon_d\equiv-\SI{0.5}{\radian\per\second}\).
The initial conditions are
\(x_0=[-1,-1.4,1.0198]^\top\),
\(\tau_0=[-0.8595,0.329,-0.3910]^\top\), and \(z(0)=0.5\),
with \(k_{11}=10\), \(k_{12}=5\), and \(k_2=0.5\).
Fig.~\ref{fig:dyn} illustrates the closed-loop response. Motion snapshots
are shown in Fig.~\ref{fig:snapshots_dyn}, while
Fig.~\ref{fig:output_dyn} shows the output signals. The third trace is the
shifted error \(1+e_2\), where
\(e_2=\dot h_2-\upsilon_d=\dot h_2+0.5\); hence its convergence to \(1\)
is equivalent to
\(\dot h_2\to-\SI{0.5}{\radian\per\second}\).

\begin{figure}[t]
    \centering
    \begin{subfigure}[b]{0.48\linewidth}
        \centering
        \includegraphics[width=\linewidth]{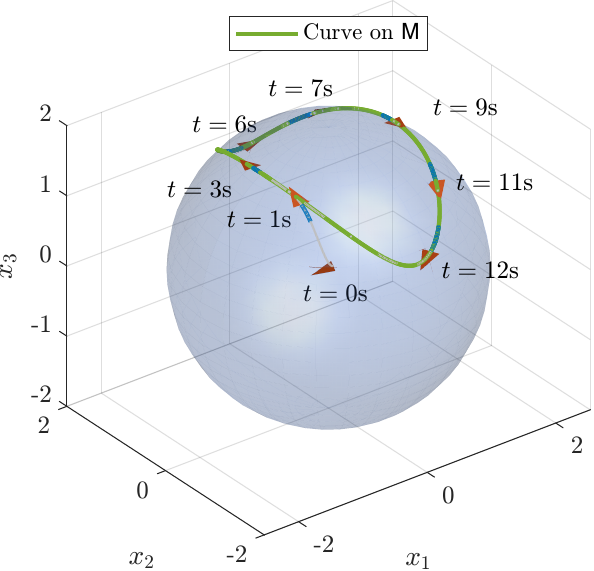}
        \caption{Sequential position and orientation.}
        \label{fig:snapshots_dyn}
    \end{subfigure}
    \hfill
    \begin{subfigure}[b]{0.48\linewidth}
        \centering
        \includegraphics[width=\linewidth]{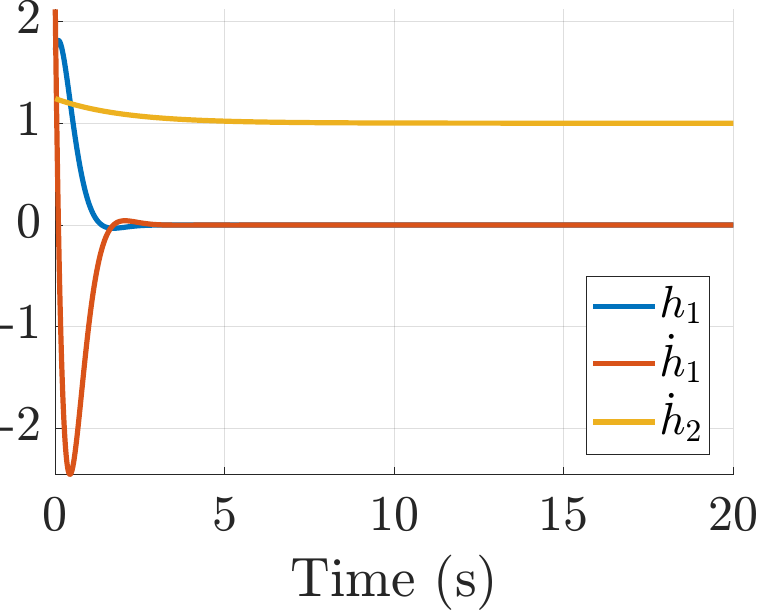}
        \caption{Output signals showing exponential convergence.}
        \label{fig:output_dyn}
    \end{subfigure}

\caption{Variable-speed path following on
\(\mathsf{M}=\{x\in\R^3:\norm{x}^2-4=0\}\).
In (b), \(1+e_2\), with \(e_2=\dot h_2-\upsilon_d\), converges to \(1\).}
    \label{fig:dyn}
\end{figure}

\section{Conclusion}
This paper developed a geometric framework for path following of a
kinematic unicycle evolving on a smooth embedded manifold. By lifting the
target path to the unit tangent bundle, path following was formulated as
stabilization of the resulting path-following manifold. For fixed
translational speed, a static feedback law renders this manifold forward
invariant and locally exponentially stable. For variable speed, a dynamic
extension renders the corresponding extended manifold forward invariant
and locally exponentially stable while exponentially regulating a
prescribed longitudinal coordinate-rate law on the admissible regularity
domain. A numerical simulation illustrates the variable-speed controller.



\bibliographystyle{IEEEtran}
\bibliography{cleanbib}

\end{document}